\documentclass[12pt]{amsart}

\usepackage{
amsfonts,
latexsym,
amssymb,
}

\newcommand{\labbel}{\label}

\newtheorem{theorem}{Theorem}[section]
\newtheorem{lemma}[theorem]{Lemma}

\newtheorem{proposition}[theorem]{Proposition} 
 
\newtheorem{corollary}[theorem]{Corollary}

\theoremstyle{definition}
\newtheorem{definition}[theorem]{Definition}

\theoremstyle{remark}
\newtheorem{remark}[theorem]{Remark}
 
\newtheorem{example}[theorem]{Example}

\newcommand{\brfrt}{\hspace{0 pt}}

\DeclareMathOperator{\cf}{cf}
\DeclareMathOperator{\CAP}{CAP}

\begin{document}
 
\title[A product theorem for accumulation properties]{A product theorem for a general accumulation property}

\author{Paolo Lipparini} 
\address{Dipartimento di Matematica\\Viale della MaStellaG Scientifica\\II Universit\`a di Roma (Tor Vergata)\\I-00133 ROME ITALY}
\urladdr{http://www.mat.uniroma2.it/\textasciitilde lipparin}




\begin{abstract}
We characterize the situations in which certain accumulation properties 
of topological spaces 
are preserved under
taking products.
\end{abstract} 
 
\maketitle

\section{Introduction} \labbel{intro}

In \cite{mostgen} we introduced the following general notion of
accumulation relative to some set $E$,
and proved that it is connected with covering properties.
If $I$ is a set, $ \mathcal S (I)$ denotes the set of all subsets of $I$.

\begin{definition} \labbel{seq}
Let $I$ be a set,  $E$ be a subset of $ \mathcal S(I)$,
and $(x_i) _{i \in I} $ 
be an $I$-indexed sequence  of elements of some topological space $X$. 

We say that a point $x \in X$ is an
\emph{accumulation point in the sense of} $E$,
or simply an \emph{$E$-accumulation point},  
 of the sequence
$(x_i) _{i \in I} $ 
 if and only if,  
for every open neighborhood $U$
of $x$, the set $\{ i \in I \mid x_i \in U  \} $ belongs to $ E$.

We say that $X$ satisfies the \emph{$E$-accumulation property}
(or that $X$ is \emph{$E$-compact}) 
if and only if every $I$-indexed sequence of elements of $X$
has some (not necessarily unique) accumulation point in the sense of $E$.
 \end{definition}

The definition is motivated by a number of examples. 

As the simplest possible non trivial case, if $ I= \omega$,
and $E$ is the set of all cofinite subsets of $ \omega$, then
$x $ is an
 an $E$-accumulation point 
 of 
$(x_n) _{n \in \omega } $ in the sense of Definition \ref{seq} 
 if and only if
$(x_n) _{n \in \omega } $ 
converges to $x$.

As a slightly more involved example, if $E=D$ is an ultrafilter over $I$, then $x$ is
a $D$-accumulation point
 of the sequence
$(x_i) _{i \in I} $ 
if and only if  
$(x_i) _{i \in I} $
\emph{$D$-converges} to $x$. Thus, in the particular case 
of an ultrafilter, 
the $D$-accumulation property  is 
the usual notion of \emph{$D$-compactness}.

In the above examples $E$ is a filter, 
a condition which, as we shall show, will turn out to provide a dividing
line between very different situations.
We now shall provide an example in which $E$ is not a filter.

\begin{example} \labbel{ex1}   
Suppose that  $| I|= \lambda $ is an infinite regular cardinal, and $E$ 
is the set of all subsets of $I$ having cardinality $\lambda$.
In this case the $E$-accumulation property turns out to be equivalent to a
property usually called
 $\CAP_ \lambda $.
A topological space $X$ 
is said to satisfy $\CAP_ \lambda $ if and only if 
every subset
$Y \subseteq X$ of cardinality $\lambda$ 
has some \emph{complete accumulation point}, that is, 
a point $x \in X$  such that 
$| Y \cap U|= \lambda $, for every neighborhood $U$ of $x$ in $X$.  
\end{example} 

See \cite{mostgen} for a complete discussion 
about $E$-accumulation and its connection with 
other compactness properties.

In this paper we discuss  conditions under which the 
$E$-accumulation property is productive (that is, preserved 
under taking arbitrary products).

We assume no separation axiom.
If not otherwise specified, when we speak of a product of 
a family of topological spaces,
we always mean the Tychonoff product.

\section{Productivity} \labbel{productivity} 

As usual, by a \emph{filter} on some set $I$, we mean
a \emph{nonempty proper} subset  $E $ of $  \mathcal S(I) $
which is closed under finite intersections and under taking supersets.  

\begin{proposition} \labbel{prop}
Let $I, J$ be  sets,  $E$ be a subset of $ \mathcal S(I)$,
and suppose that
 $(X_j) _{j \in J} $ is a sequence of topological spaces.
Consider the following conditions.
  \begin{enumerate}    
\item
The Tychonoff product $ \prod _{ j\in J} X_j$ 
  satisfies the $E$-accumulation property.
\item 
For every $j \in J$, $X_j$ satisfies the $E$-accumulation property.
  \end{enumerate} 

Condition (1) implies Condition (2).

If $E$ is a filter,
then 
Condition (2) implies Condition (1).
 \end{proposition} 

In fact, a local version of the result holds.
Proposition \ref{prop} is immediate from the following lemma.

\begin{lemma} \labbel{lempr}
Let $I, J$ be  sets,  $E$ be a subset of $ \mathcal S(I)$,
and suppose that
 $(X_j) _{j \in J} $ is a sequence of topological spaces.
Suppose that $ (x _{j,i} ) _{j \in J, i \in I} $ is an
$I$-indexed sequence
 of elements of $\prod _{ j\in J} X_j$,
and that 
$z= (z_j) _{j \in J}  $ is an element of $  \prod _{ j\in J} X_j$. If
  \begin{enumerate}    \item   
$z $ is an
$E$-accumulation point of 
$ (x _{j,i} ) _{j \in J, i \in I} $
in $\prod _{ j\in J} X_j$,
\end{enumerate} 

then
  \begin{enumerate}    \item[(2)]   
for every $\bar\jmath \in J$, the 
  element $z _{\bar\jmath}  \in  X _{\bar\jmath}$ is an
$E$-accumulation point of the
$I$-indexed sequence
$ (x _{\bar\jmath,i} ) _{ i \in I} $ of elements
of $X _{\bar\jmath}$.
\end{enumerate} 

If $E$ is a filter, then (2) implies (1).
 \end{lemma}  

 \begin{proof}
Suppose that (1) holds.
For each $\bar\jmath \in J$, and 
for every open neighborhood $U  $ of 
$z_{\bar\jmath}$ in   $X_{\bar\jmath}$, 
consider the open neighborhood 
$V $ of
$z$  in $\prod _{ j\in J} X_j$
defined by   
$V = \prod _{ j\in J} V_j $,
where $V_j= X_j$,
if $j \not= \bar\jmath$,
and $V _{\bar\jmath} = U $.   

Then 
$\{ i \in I \mid x _{\bar\jmath,i} \in U  \text{ in $X _{\bar\jmath} $} \} 
=
\{ i \in I \mid (x _{j,i}) _{j \in J}  \in V \text{ in $\prod _{ j\in J} X_j$}  \}
\in E
$,
since 
$z $ is an
$E$-accumulation point of 
$ (x _{j,i} ) _{j \in J, i \in I} $ in $\prod _{ j\in J} X_j$.
 
Conversely, suppose that $E$ is a filter,
and that (2) holds. Let 
$U$ be an open  neighborhood of $z$ in  
$\prod _{ j\in J} X_j$.
Then there is an open neighborhood 
$U'$ of $z$
such that 
$U' \subseteq U$ and 
$U' = \prod _{ j\in J} U_j $,
where, for every $j \in J$, 
$U_j$ is an open neighborhood of $z_j$ in $X_j$, 
and there is some finite $F \subseteq J$ such that 
$U_j = X_j$, for every $j \in J \setminus F$.

By (2), for every $j \in F$,  
$e_j=\{ i \in I \mid x _{j,i}  \in U_j  \} \in  E$,
hence $\{ i \in I \mid (x _{j,i}) _{j\in J}   \in U'  \}= \bigcap _{j \in F} e_j \in E $,
since $E$ is closed under finite intersection.
Since  $U' \subseteq U$, then 
$\{ i \in I \mid (x _{j,i}) _{j\in J}   \in U'  \} \subseteq 
\{ i \in I \mid (x _{j,i}) _{j\in J}   \in U  \} \in E$,
since $E$ is closed under supersets.

Since $U$ was an arbitrary open neighborhood of $z$ in $\prod _{ j\in J} X_j$,
the above argument shows that  
$z $ is an
$E$-accumulation point of 
$ (x _{j,i} ) _{j \in J, i \in I} $ in $\prod _{ j\in J} X_j$.
\end{proof}

As the example of complete accumulation points
shows, restricting Definition \ref{seq} only to filters causes a
big loss in generality. However we can show that, in some local sense,
 we get an equivalent definition  when we let $E$ vary only among filters.
This fact is the main key for  applications to products.

\begin{lemma} \labbel{local}
Suppose that  $I$ is a set,  $E$ is a subset of $ \mathcal S(I)$, and
$X$ is a topological space.
Let $ (x _{i} ) _{ i \in I} $
 be some given
$I$-indexed sequence
 of elements of $ X$, and suppose that 
$x \in X$ is an $E$-accumulation point of
$ (x _{i} ) _{ i \in I} $.

If, for every neighborhood $U$ of $x$ in $X$, we put 
$f_U = \{ i \in I \mid x_i \in U \}$,
then 
$F = \{ f_U \mid U \text{ is an open neighborhood of $x$ in } $X$ \}$  
is closed under finite intersections, $F \subseteq E$, and $x $ is an $F$-accumulation point of
$ (x _{i} ) _{ i \in I} $.  
\end{lemma} 

\begin{proof}
That $ F \subseteq E$ is immediate from Definition \ref{seq},
and the assumption that $x $ is an $E$-accumulation point of
$ (x _{i} ) _{ i \in I} $.

If $U$ and $V$ are open neighborhoods of $x$, then
$U \cap V$ is an open neighborhood of $x$, 
and $f_U \cap f_V  = f _{U \cap V} \in F $,
hence $F$  is closed under intersections. 

By the definition of $F$, $x $ is an $F$-accumulation point of
$ (x _{i} ) _{ i \in I} $.
\end{proof}

Lemma \ref{local} may be seen as a general version of the fact that
an element $x \in X$ belongs to the closure of some set $Y \subseteq X$ if and only if 
there is some net of elements of $Y$ converging to $x$ \cite[Proposition 1.6.3]{E},
or, in equivalent form, if and only if there exists a filter-base
consisting of subsets of $Y$ converging to $x$ \cite[Proposition 1.6.9]{E}.
Indeed, this latter result is a consequence of the particular case 
of Lemma \ref{local} when $I= Y$ and
 $E= \mathcal S (I) \setminus \{ \emptyset   \} $.
 
In a rough but suggestive sense, Lemma \ref{local} asserts that, locally,
``accumulation'' becomes a form of ``convergence'' (an arbitrary
 $E$ corresponds to some notion of accumulation, but a \emph{filter} $F$ 
corresponds to a notion of convergence).
Since most  notions of convergence are productive 
(the last statement in Proposition \ref{prop}, as far as we are concerned), and
since any failure of accumulation implies some local failure,  
we get that an \emph{accumulation} 
property is productive,  for some given space $X$, if and only if
$X$ satisfies some stronger notion of \emph{convergence}.

\begin{theorem} \labbel{prod} 
Suppose that $X$ is a topological space, $I$ is a set, and $E$ is a subset of 
$ \mathcal S(I)$ closed under taking supersets.
Then the following conditions are equivalent.
  \begin{enumerate}
\item 
Every power of $X$ satisfies the $E$-accumulation property.
\item
$X$ satisfies the $F$-accumulation property,
for some filter $F \subseteq E$.
\item
$X$ satisfies the $F$-accumulation property,
for some filter $F $ which is  maximal among the filters
contained in $E$.
\item
$X^ \delta $ satisfies the $E$-accumulation property, where $\delta$ 
is the cardinality of the set of all the 
filters which are  maximal among the filters
contained in $E$.
\item
$X^ \delta $ satisfies the $E$-accumulation property, for $\delta =| X| ^{|I|} $. 
  \end{enumerate}
\end{theorem}

 \begin{proof}
(1) $\Rightarrow $  (4), (1) $\Rightarrow $  (5) and (3) $\Rightarrow $  (2) are trivial.
We shall complete the proof  by proving 
(2) $\Rightarrow $  (1), (4) $\Rightarrow $  (3) and (5) $\Rightarrow $  (2).

(2) $\Rightarrow $  (1) By the last statement in Proposition \ref{prop},
every power of $X$ satisfies the $F$-accumulation property. Since 
$F \subseteq E$,  the    $E$-accumulation property follows trivially.

In order to prove that (4) implies (3), we shall suppose that (4) holds and
(3) fails, and we shall reach a contradiction. If (3) fails, 
then, for every maximal filter $F \subseteq E$, there 
is a sequence $ (x _{i} ) _{ i \in I} $ such that
no $x \in X$ is an $F$-accumulation point of 
$ (x _{i} ) _{ i \in I} $.

Enumerate all maximal filters $ \subseteq E$
as $(F_j) _{j \in \delta } $, and, for each $ j \in \delta $,
choose a sequence      
$ (x _{j,i} ) _{ i \in I} $
which has no $F_j$-accumulation point in $X$.
  
Consider $ (x _{j,i} ) _{  j \in \delta , i \in I} $
as a single $I$-indexed sequence of elements of $X^ \delta $.
By (4),  $ (x _{j,i} ) _{  j \in \delta , i \in I} $ has some 
$E$-accumulation point $z= (z _j) _{j \in \delta } $
in $X^ \delta $.  
By applying Lemma \ref{local} to 
$ (x _{j,i} ) _{  j \in \delta , i \in I} $ and $z$ in  $X^ \delta $,
we get some $F' \subseteq E$ such that $F'$ is closed 
under finite intersections,
and $z$ is an $F'$-accumulation point of $ (x _{j,i} ) _{  j \in \delta , i \in I} $. 
If $F''$ is the filter generated by $F'$, then
also $F'' \subseteq E$,     since $E$ is closed under taking supersets.
By an easy application of Zorn's
Lemma, there is a filter $F$  such that $F'' \subseteq F \subseteq E$,
and $F$  is maximal among the filters contained in $E$.  
Since $F' \subseteq F$, then  $z$ is an $F$-accumulation point of $ (x _{j,i} ) _{  j \in \delta , i \in I} $.  

By the construction of the sequence $(F_j) _{j \in \delta } $, there is $\bar\jmath \in \delta $
such that $F=F _{\bar\jmath} $.  By Lemma \ref{lempr} (1) $\Rightarrow $  (2),
 $z_{\bar\jmath}$ is an $F$-accumulation point of $ (x _{\bar\jmath,i} ) _{  i \in I} $ 
in $X$. This contradicts the choice of the sequence $ (x _{\bar\jmath,i} ) _{  i \in I} $.
We have reached a contradiction, thus (4) $\Rightarrow $  (3) is proved.
 
The proof that (5) implies (2) is rather similar, and somewhat simpler. 
We shall suppose that (5) holds and that
(2) fails, and we shall reach a contradiction. 
This time,
enumerate all 
$I$-indexed sequences of elements of $X$,
as
$ (x _{j,i} ) _{  j \in \delta , i \in I} $.
This is possible since there are exactly
$\delta$-many such sequences.

By (5), the single sequence
 $ (x _{j,i} ) _{  j \in \delta , i \in I} $ in $X^ \delta $
 has some 
$E$-\brfrt accumulation point $z= (z _j) _{j \in \delta } $
in $X^ \delta $.  
By Lemma \ref{local}, and since $E$ is closed under supersets, 
there is a filter $F \subseteq E$ such that    $z$ is an $F$-accumulation point of $ (x _{j,i} ) _{  j \in \delta , i \in I} $.  
Since (2) fails,
 there 
is a sequence $ (x _{i} ) _{ i \in I} $ such that
no $x \in X$ is an $F$-accumulation point of 
$ (x _{i} ) _{ i \in I} $.
Since  $ (x _{j,i} ) _{  j \in \delta , i \in I} $  enumerates
all possible $I$-indexed sequences of elements of $X$,
there is $\bar\jmath \in \delta $ such that 
$ x _{\bar\jmath ,i}= x_i $, for every $i \in I$, thus 
 $ (x _{\bar\jmath ,i} ) _{ i \in I} $ 
witnesses the failure of the $F$-accumulation property 
for $X$.
However, 
 Lemma \ref{lempr} (1) $\Rightarrow $  (2),
with $F$  in place of $E$, implies that
 $z_{\bar\jmath}$ is an $F$-accumulation point of $ (x _{\bar\jmath,i} ) _{  i \in I} $, a contradiction.
 \end{proof} 

By Proposition \ref{esmall} in Section \ref{concl} below, if 
Condition (2) in Theorem \ref{prod} holds, then $F$  is actually
an ultrafilter over $I$, except for very special cases. This fact might explain 
 the necessity of the use of ultrafilters in many product theorems.

The simplest case in Theorem \ref{prod}
is of course when $E$ itself is a filter, in which case we get that
the $E$-accumulation property is productive, 
 a fact which follows already from Proposition \ref{prop}.

The case of complete accumulation points is more significant.
As a corollary
of Theorem \ref{prod}, we get a result by Ginsburg  and Saks \cite{GS,Sa}, asserting that, for every topological space $X$, 
every power of $X$ satisfies $\CAP_ \lambda $
if and only if 
$X$ is $D$-compact, for some ultrafilter $D$ uniform over $\lambda$.
 We shall prove a slightly more general version which deals with ordinals.
It is not completely clear how much significant this improvement is,
however we include it, since its proof involves no new technical detail.

\begin{definition} \labbel{capa} 
If $\alpha$ is an ordinal,  $X$ is a topological space, 
and $( x_ \gamma ) _{ \gamma \in \alpha } $ is a sequence of elements of $X$,
we say 
that $x \in X$ is an
\emph{$\alpha$-complete accumulation point} of 
$( x_ \gamma ) _{ \gamma \in \alpha } $
if and only if, for every neighborhood $U$ of $x$,
the set $\{ \gamma \in \alpha \mid x_ \gamma \in U\}$ 
has order type $\alpha$.

We say that $X$ satisfies
the \emph{$\alpha$-complete accumulation property}, 
$ \CAP ^*_{ \alpha } $ for short, if and only if 
every $\alpha$-indexed sequence $( x_ \gamma ) _{ \gamma \in \alpha } $
of elements of $X$ has some $\alpha$-complete accumulation point in $X$.

It is easy to see that, if $\alpha$ is a \emph{regular} cardinal, then
the starred $ \CAP^* _{ \alpha } $ 
 turns out to be
equivalent to the unstarred $ \CAP _{ \alpha } $ 
of Example \ref{ex1}.
If $\alpha$ is a singular cardinal, then $ \CAP^* _{ \alpha } $ 
is
equivalent to the conjunction of $ \CAP _{ \alpha } $ and 
$ \CAP _{ \cf\alpha } $. See \cite[Proposition 3.3]{tproc2}.  
\end{definition}   

\begin{remark} \labbel{part}   
Notice that the notions in the above definition are the particular cases
of Definition \ref{seq} when $I=\alpha$  
and $E$ is the set of all subsets of $\alpha$ having order type $\alpha$.
\end{remark} 

\begin{definition} \labbel{ordunif}
As usual, a filter $F$ over some set $I$ is said to be \emph{uniform}
if and only if every element of $F$ has cardinality $| I|$. We generalize this notion to ordinals.

If $\alpha$ is an ordinal, we say that a filter $F$ over $\alpha$ 
is \emph{ordinal-uniform} if and only if every
member of $F$ has order type $\alpha$. 
 \end{definition}   

\begin{remark} \labbel{unifordcard}  
If
$\alpha$ is an infinite cardinal, then $F$ over $\alpha$ is uniform
if and only if it is ordinal-uniform. When $\alpha$ is not a cardinal,
the two notions are distinct. Ordinal-uniformity implies uniformity, but
the converse does not necessarily hold.
 \end{remark}  

As far as ordinal-uniformity is concerned, there is a strong dichotomy
according to the form of the ordinal on which a filter is
taken.

\begin{proposition} \labbel{dich}  
(1) If $\alpha= \omega ^ \eta $ (ordinal exponentiation), for some ordinal $\eta$ (in particular, when $\alpha$ 
is an infinite cardinal), then  any filter ordinal-uniform  over $\alpha$
can be extended to an ultrafilter  ordinal-uniform  over $\alpha$.

(2) If $\alpha$ has not the form $ \omega ^ \eta $, for some $\eta$,
then there is no  ordinal-uniform  ultrafilter over $\alpha$.
\end{proposition}

 \begin{proof} 
(1) We show that if $\alpha= \omega ^ \eta $, $F$  is a filter,
every element of $F$  has order type $\alpha$, and  
$G= \{ Z \subseteq \alpha \mid  \alpha \setminus Z \text{ has order type }
< \alpha \}$, 
then 
$F \cup  G$ has the finite intersection property.  
Indeed, due to the particular form of $\alpha$,  the union of a finite number
of subsets of $\alpha$ of order type $<\alpha$ has order type $<\alpha$, hence 
$G$ is closed under finite intersections. Moreover, the intersection
of an element of $F$  and an element of $G$ is nonempty (in this case, the particular form
of $\alpha$ plays no role).
Hence $F \cup  G$ can be extended to an ultrafilter $D$ over $\alpha$, 
and $D$ is ordinal-uniform, since $D$ extends $G$.

(2)
There are 
disjoint $Z_1 , Z_2 \subseteq  \alpha $ such that $Z_1 \cup Z_2 = \alpha $, and
both $Z_1 $ and $  Z_2 $ have order type $<  \alpha $ (just express 
$\alpha$ in normal form, and consider it as a sum of two summands). 
Since any ultrafilter $D$ 
over $\alpha$ must contain one of $Z_1 $ or $  Z_2 $, then $D$ is not ordinal-uniform.
\end{proof} 

\begin{corollary} \labbel{saks}
For every topological space $X$ and every ordinal $\alpha$,
the following conditions are equivalent.
  \begin{enumerate}    
\item  
Every power of $X$ satisfies $\CAP^*_ \alpha  $.
\item
$X$ is $F$-compact, for some filter $F$ ordinal-uniform over $ \alpha $.
\item
$X^ \delta $ satisfies $\CAP^*_ \alpha $, for $\delta= \min \{ 2 ^{ 2 ^{| \alpha |} }  , |X| ^{| \alpha |}  \}  $.
 \end{enumerate}  

If, in addition, $\alpha= \omega ^ \eta $ (ordinal exponentiation), for
some ordinal $\eta$,
then the preceding conditions are also equivalent to:
  \begin{enumerate}    
\item[(4)] 
$X$ is $D$-compact, for some ultrafilter $D$ ordinal-uniform over $ \alpha $.
  \end{enumerate} 
 \end{corollary}

 \begin{proof} 
The Corollary  is immediate from Theorem \ref{prod}, 
 by taking $E$ to be the set of all subsets of $\alpha$ of order type $\alpha$, and
using Remarks \ref{part} and \ref{unifordcard} and Proposition \ref{dich}.
\end{proof}

The case $\alpha= \omega $ of Corollary \ref{saks} is due
to Ginsburg  and Saks \cite[Theorem 2.6]{GS}, 
via the equivalence between $ \CAP^* _{ \omega } $   
and countable compactness. For $\alpha$ an arbitrary \emph{cardinal}, 
the equivalence of (1) and (2) in Corollary \ref{saks} appears in
\cite[Theorem 6.2]{Sa}. Related results appear in 
 \cite[Theorem 5.6]{ScSt} and \cite[Corollary 2.15]{GF1}.

\section{Concluding remarks} \labbel{concl}

Many generalizations of our results are possible, 
and involve very little new technical tools.

For example, we can introduce a notion relative to subspaces.

\begin{definition} \labbel{seqy}
Let $I$ be a set,  $E$ be a subset of $ \mathcal S(I)$, and $Y$ be a subspace of 
some topological space $X$.

We say that \emph{$Y$ satisfies the $E$-accumulation property in $X$}
if and only if, for  every $I$-indexed sequence $(y_i) _{i \in I} $ 
of elements of $Y$,
there is $x \in X$ such that, in $X$,
$x$ is an  accumulation point of 
$(y_i) _{i \in I} $ in the sense of $E$.

If we take $Y=X$, we reduce to the case dealt in the last statement of Definition \ref{seq}.
 \end{definition}

We now show that, for ``small'' $E$, the   $E$-accumulation
property is trivially false, for $T_1$ spaces with more than one point.

\begin{proposition} \labbel{esmall}
Suppose that $I$ is a set, $E \subseteq \mathcal S (I)$,
and there is a set $J \subseteq I$ such that 
neither $J$ nor $I \setminus J$ belong to $E$.
In particular, this assumption applies when $E$ is a filter on $I$ 
which is not an ultrafilter.

If some topological space $X$ satisfies the   
$E$-accumulation
property, then, for every pair
$x, y$ of distinct points of $X$, there is some point 
$z $ of $ X$ such that every neighborhood $U$ of $z $ 
contains both $x$ and $y$. 
 \end{proposition}  

\begin{proof}
Suppose that $x \not=y \in X$.
Consider the sequence defined by 
$x_i= x$, if $i \in J$, and
 $x_i= y$, if $i \in I \setminus J$,
and apply the $E$-accumulation
property to this sequence.
 \end{proof} 

In particular, if $X$ is a $T_1$ topological space,
$| X| > 1$ and $X$ satisfies $ \CAP ^*_{ \alpha } $,
then $\alpha = \omega ^ \eta$ (ordinal exponentiation),
for some ordinal $\eta$. Otherwise (for example, expressing 
$\alpha$ in normal form) $\alpha$ is the union of two disjoint
subsets each of order type $<\alpha$, and we can apply Proposition \ref{esmall}.

Thus, (full) compactness does not necessarily imply 
$ \CAP ^*_{ \alpha } $, for every ordinal $\alpha$,
though it is well known that it implies $ \CAP _{ \lambda} $, for every infinite
\emph{cardinal} $\lambda$.

We can also give a two-ordinals generalization of Definition \ref{capa}.

\begin{definition} \labbel{capab} 
If $ \beta , \alpha$ are ordinals,  $X$ is a topological space, 
and $( x_ \gamma ) _{ \gamma \in \alpha } $ is a sequence of elements of $X$,
we say 
that $x \in X$ is a
\emph{$\beta$-$\alpha$-complete accumulation point} of 
$( x_ \gamma ) _{ \gamma \in \alpha } $
if and only if, for every neighborhood $U$ of $x$,
the set $\{ \gamma \in \alpha \mid x_ \gamma \in U\}$ 
has order type $\geq \beta $.

We say that $X$ satisfies
the \emph{$\beta$-$\alpha$-complete accumulation property}, 
 if and only if 
every $\alpha$-indexed sequence $( x_ \gamma ) _{ \gamma \in \alpha } $
of elements of $X$ has some $ \beta $-$ \alpha$-complete accumulation point in $X$.
\end{definition}



\end{document}